\documentclass[a4paper]{amsart}
\usepackage{amssymb,amscd}
\usepackage{graphicx}
\usepackage[all]{xy}

\input xypic

\theoremstyle{problems}

\newtheorem{theorem}{Theorem}[section]

\newtheorem{lemma}[theorem]{Lemma}

\CompileMatrices

\theoremstyle{definition}

\theoremstyle{remark}

\numberwithin{equation}{section}

\renewcommand{\phi}{\varphi}
\newcommand{\bl}{\mbox{$\lambda\kern-0.53em\lambda$}}
\newcommand{\bmu}{\mbox{$\mu\kern-0.55em\mu$}}
\newcommand{\bnu}{\mbox{$\nu\kern-0.51em\nu$}}
\def\bphi{\mbox{$\varphi\kern-0.59em\varphi$}}

\def\C{\mathbb C}
\def\Di{\mathbb D}

\renewcommand{\le}{\leqslant}

\def\пїЅпїЅ{\symbol{"BE}}
\def\пїЅпїЅ{\symbol{"BF}}

\newcommand{\Tor}{\mathop{\rm Tor}\nolimits}

\newcommand{\zk}{\mathcal Z_{\mathcal K}}

\begin{document}

\title{Pontryagin algebras of some moment-angle-complexes}
\author{Yakov Veryovkin} 
\address{Departments of Mathematics and Mechanics, Moscow State University.}
\email{verevkin\_j.a@mail.ru}

\thanks{The research was carried out at the Steklov Institute of Mathematics RAS and supported by the Russian Science Foundation (project no.~14-11-00414).}


\maketitle

\section{Introduction}
In this paper we consider the problem of describing the Pontryagin algebra (loop homology) of moment-angle complexes and manifolds. The moment-angle complex $ \mathcal Z_\mathcal{K} $ is a cell complex built of products of polydiscs and tori parametrised by simplices in a finite simplicial complex $\mathcal K$. It has a natural torus action and plays an important role in toric topology~\cite {bupa}. In the case when $ \mathcal K $ is a triangulation of a sphere, $ \mathcal Z_\mathcal {K} $ is a topological manifold, which has interesting geometric structures.

Generators of the Pontryagin algebra $ H _* (\varOmega \mathcal Z_\mathcal {K}) $ when $ \mathcal K $ is a flag complex were described in~\cite{g-p-t-w12}. Describing relations is often a difficult problem, even when $\mathcal{K}$ has a few vertices. Here we describe these relations in the case when $\mathcal K$ is the boundary of pentagon or hexagon. In this case, it is known that $ \mathcal Z_\mathcal {K} $ is a connected sum of products of spheres with two spheres in each product (see McGavran~\cite{mcga79}, and also~\cite{gi-lo09},~\cite{bo-me06}). Therefore $ H _* (\varOmega \mathcal Z_\mathcal {K}) $ is a one-relator algebra and we describe this one relation explicitly, therefore giving a new homotopy-theoretical proof of McGavran's result. An interesting feature of our relation is that it includes iterated Whitehead products, which vanish under the Hurewicz homomorphism. Therefore, the form of this relation cannot be deduced solely from the result of McGavran.

I wish to express gratitude to my advisor Professor Taras Evgenievich Panov for his help, attention and invaluable advice. I also thank Ivan Limonchenko for stimulation discussions.

\section{Preliminaries}

Let $ \mathcal K $ be a \emph{simplicial complex} on the set $ [m] = \{1,2,3, \dots, m \} $, i.e. $ \mathcal K $ is a collection of subsets $ I \subset [m] $ closed under inclusion. Subsets $ I \in \mathcal K $ are called \emph{simplices}. Simplicial complex is \emph{flag} if any minimal non-simplex consists of two elements. In other words, $ \mathcal K $ is a flag complex if every set of vertices pairwise connected by edges generates a simplex.

For any $ I \subset [m] $ define the \emph{full subcomplex} (the restriction of $ \mathcal {K} $ to $ I $):
\[
\mathcal{K}_I = \{J \in \mathcal K \colon J \subset I\}.
\]

Let $ k $ be a commutative ring with unit. The \emph{face ring} of a simplicial complex $ \mathcal K $ is the quotient of the graded polynomial ring  $k[v_1, \ldots, v_m]$ by the monomial ideal generated by non-simplices:
\[
k[\mathcal K] = k[v_1, \ldots, v_m] \big/ \bigl(v_{i_1}\!\ldots v_{i_k} \colon \{i_1, \ldots, i_k\} \notin \mathcal K \bigr),\quad \deg v_i = 2.
\]


Consider the unit polydisc in the complex space:
\[
  \Di^m=\bigl\{ (z_1,\ldots,z_m)\in\C^m\colon |z_i|^2\le1,\quad i=1,\ldots,m
  \bigr\}.
\]
For each subset $I = \{i_1, \cdots, i_k\} \subset [m]$ define the following subset of the polydisc:
\[
  B_I=\bigl\{(z_1,\ldots,z_m)\in\Di^m\colon |z_i|^2=1
 \quad \text{if }i\notin I \bigr\}
\]

The \emph{moment-angle complex} corresponding to the simplicial complex $ \mathcal K $ is the space
\[
  \zk=\bigcup_{I \in\mathcal K}B_I\subset\Di^m,\\
\]
where the union is taken in $ \Di ^ m $. Similarly, consider the subset:
\[
(\mathbb{C}P^\infty)^\mathcal{K}=\bigcup_{I \in \mathcal K}BT^I \subset BT^m
\]
where $BT^m=(\mathbb{C}P^\infty)^m$, $ BT^I=\bigl\{(x_1,\ldots,x_m)\in BT^m\colon x_i=*
  \text{ if }i\notin I\bigr\}$.

We consider homology and cohomology with coefficients in the ring $k$ and do not specify it in the notation.

\begin{theorem}[\cite{bupa}]\label{hzk}
There are isomorphisms of rings:
\[
\begin{aligned}
H^*((\mathbb{C}P^\infty)^\mathcal{K}) & \cong k[\mathcal{K}],\\
H^*(\mathcal{Z}_\mathcal{K}) & \cong \Tor_{k[v_1, \ldots, v_m]}(k[\mathcal{K}], k) \cong H\bigl{(}\Lambda [u_1, \ldots, u_m] \otimes k[\mathcal{K}], d \bigr{)},\\
H^p(\zk) & \cong \bigoplus_{I \subset [m]} \widetilde{H}^{p - |I| - 1} (\mathcal{K}_I),
\end{aligned}
\]
where $k[\mathcal{K}]$ is the face ring, $du_i = v_i,\; dv_i = 0, \quad \mathop{\mathrm{deg}} u_i = 1, \; \mathop{\mathrm{deg}} v_i = 2$.
\end{theorem}

\begin{theorem}[\cite{bupa}]
The moment-angle complex $ \mathcal Z_ \mathcal {K} $ is the homotopy fibre of the inclusion $(\mathbb{C}P^\infty)^\mathcal{K} \rightarrow (\mathbb{C}P^\infty)^m$.
\end{theorem}

The homotopy fibration $\mathcal Z_\mathcal{K} \rightarrow (\mathbb{C}P^\infty)^\mathcal{K} \rightarrow (\mathbb{C}P^\infty)^m$ gives rise to the exact sequence of Pontryagin algebras:
\[
1 \rightarrow H_* (\varOmega \mathcal{Z}_\mathcal{K}) \rightarrow H_* \bigl( \varOmega (\mathbb{C}P^\infty)^\mathcal{K}\bigl) \rightarrow \Lambda[u_1, \dots, u_m] \rightarrow 0,
\]
and $ \Lambda [u_1, \dots, u_m] = H _ * (\varOmega (\mathbb{C} P ^ \infty)^m) = H_* (T^m) $. Therefore, $H_*(\varOmega \mathcal{Z}_\mathcal{K})$ is the commutator subalgebra of the non-commutative algebra $H_* \bigl( \varOmega (\mathbb{C}P^\infty)^\mathcal{K}\bigl)$. This algebra can be described explicitly in the case when $ \mathcal K $ be a flag complex:

\begin{theorem}[\cite{bupa}]\label{gptw12}
Let $\mathcal K$ be a flag complex. Then
\[
H_*(\varOmega (\mathbb{C}P^\infty)^\mathcal{K}; k) \cong T \langle \mu_1, \dots, \mu_m \rangle / (\mu_i^2 = 0, \mu_i \mu_j + \mu_j \mu_i = 0 \text{ if } \{i, j\} \in \mathcal K),
\]
where $T \langle \mu_1, \dots, \mu_m \rangle$ is a free (tensor) algebra, $\deg \mu_i = 1$.
\end{theorem}

Here $\mu_i \in H_1(\varOmega( (\mathbb{C}P^\infty)^\mathcal{K}) )$ is the canonical generator corresponding to the $i$-th coordinate mapping $S^1 = \varOmega( \mathbb{C}P^\infty)^m \to \varOmega( \mathbb{C}P^\infty)^\mathcal{K}$.

The following result describes multiplicative generators of the commutator algebra $H_* (\varOmega \mathcal{Z}_\mathcal{K})$.

\begin{theorem}[\cite{g-p-t-w12}]\label{TP}
Let $\mathcal{K}$ be a flag complex. Then the subalgebra $H_* (\varOmega \mathcal{Z}_\mathcal{K}) \subset H_* \bigl( \varOmega (\mathbb{C}P^\infty)^\mathcal{K} \bigl)$ is multiplicatively generated by iterated commutators of the form:
\[
[\mu_i, \mu_j], \quad[\mu_{k_1}, [\mu_j, \mu_i]], \quad \dots,\quad [\mu_{k_1}, [\mu_{k_2},[\mu_{k_3}, \dots, [\mu_{k_{m-2}}, [\mu_j, \mu_i]]\dots]]],
\]
where $k_1 < k_2 < \dots < k_p < j > i$, $k_s \neq i$ for any $s$, and $i$ is the smallest vertex in a connected component not containing $j$ of the full subcomplex $\mathcal{K}_{\{k_1, \dots, k_p, j, i\}} \subset \mathcal{K}$. Furthemore, this multiplicative generating set is minimal, that is, the commutators above form a basis in the submodule of indecomposables in $H_*(\varOmega \zk)$.
\end{theorem}



\section{The case of pentagon}

In this section, $ \mathcal K $ is the boundary of pentagon. That is, $\mathcal{K}$ is the simplicial complex on $[5] = \{1, 2, 3, 4, 5\}$ with maximal simplices $\{1,2\}, \{2,3\}, \{3,4\}, \{4,5\}, \{5,1\}$. We describe the cohomology, homology and Pontryagin algebra of $\zk$.

\subsection*{Cohomology}


The cohomology ring structure of $\zk$ can be described using Theorem~\ref{hzk}. The nontrivial cohomology groups are
\begin{align*}
H^0 (\zk) & \cong \widetilde{H}^{-1} (\mathcal{K}_\varnothing) \cong k,\\
H^3 (\zk) & \cong \bigoplus_{|I|=2} \widetilde{H}^0 (\mathcal{K}_I) \cong k^5,\\
H^4 (\zk) & \cong \bigoplus_{|I|=3} \widetilde{H}^0 (\mathcal{K}_I) \cong k^5,\\
H^7 (\zk) & \cong \widetilde{H}^1 (\mathcal{K}) \cong k.
\end{align*}

To describe the cohomology ring it is convenient to work with Koszul algebra $\varLambda [u_1, \ldots, u_5] \otimes k[\mathcal{K}]$. A generator of $H^7(\zk)$ is represented by any monomial $u_i u_{i+1} u_{i+2} v_{i+3} v_{i+4} \in \varLambda [u_1, \ldots, u_5] \otimes k[\mathcal{K}]$, where the indices are considered modulo~$5$. All these monomials represent the same cohomology classes; this follows from the identities
\begin{multline*}
[d(u_1 u_2 u_3 u_4 v_5)] = [d(u_1 u_2 u_3 u_5 v_4)] = [d(u_1 u_2 u_4 u_5 v_3)] =\\
= [d(u_1 u_3 u_4 u_5 v_2)] = [d(u_2 u_3 u_4 u_5 v_1)] = 0,
\end{multline*}
that is,
\begin{multline}\label{equ}
[u_2 u_3 u_4 v_5 v_1] - [u_1 u_2 u_3 v_4 v_5] = [u_1 u_2 u_5 v_3 v_4] - [u_1 u_2 u_3 v_4 v_5] =\\
= - [u_1 u_4 u_5 v_2 v_3] + [u_1 u_2 u_5 v_3 v_4] = [u_3 u_4 u_5 v_1 v_2] - [u_1 u_4 u_5 v_2 v_3] =\\
= [u_3 u_4 u_5 v_1 v_2] - [u_2 u_3 u_4 v_1 v_5] = 0.
\end{multline}
We denote $ t = [u_1 u_2 u_3 v_4 v_5] \in H^7 (\zk)$ and calculate the product in the cohomology ring.
We choose bases for $H^3(\zk)$ and $H^4(\zk)$ as shown in Table~\ref{zk5prod}. For any given basis element of $H^3(\zk)$ there is a unique basis element of $H^4(\zk)$ such that the product of these two elements is $t$. The product of any other two elements of $H^3(\zk)$ and $H^4(\zk)$ is zero. For example,
\[
[u_1 v_3] \cdot [u_4 u_5 v_2] = [u_1 u_4 u_5 v_2 v_3] = [u_1 u_2 u_3 v_4 v_5] = t,
\]
where the second identity follows from~\eqref{equ}.

From these equations we obtain a table of products in cohomology
\begin{table}[h]
\caption{Cohomology classes and their products.}\label{zk5prod}
\begin{center}
\begin{tabular}{c|c|c}
$H^3(\zk)$ & $H^4(\zk)$ & Product \\
\hline
$[u_1 v_3]$ & $[u_4 u_5 v_2]$ & $t$ \\
\hline
$[u_2 v_4]$ & $-[u_1 u_5 v_3]$ & $t$ \\
\hline
$[u_3 v_5]$ & $[u_1 u_2 v_4]$ & $t$ \\
\hline
$[u_4 v_1]$ & $[u_2 u_3 v_5]$ & $t$ \\
\hline
$[u_5 v_2]$ & $[u_3 u_4 v_1]$ & $t$ \\
\end{tabular}
\end{center}
\end{table}
\subsection*{Homology}
To describe the homology of $\zk$ we use the cell decomposition from \cite[\S4.4]{bupa}. Each factor $\mathbb{D}$ in $\mathbb{D}^m$ is decomposed into three cells: the point $1 \in \mathbb{D}$ is the $0$-cell; the complement to $1$ in the boundary circle is the $1$-cell, which we denote by $S$; and the interior of $\mathbb{D}$ is the $2$-cell, which we denote by $D$. {\samepage{By taking a product we obtain a cellular decomposition of $\mathbb{D}^m$ whose cells have the form
\[
\prod_{i \in I} D_i \times \prod_{j \in J} S_j,
\]
where $I, J \subset [m]$ and $I \cap J = \varnothing$.}} The cells of $\zk \subset \mathbb{D}^m$ are specified by the condition $I \in \mathcal{K}$.

We now describe cellular cycles dual to the cohomology classes in Table~\ref{zk5prod}. These are given in Table~\ref{zk5hombas}.
For example, the cochain $ u_1 v_3 $ takes value $1$ on the chain $ S_1 D_3 + D_1 S_3 $ and vanishes on all other chains given.
\begin{table}[h]
\caption{Cellular cycles representing basis homology classes.}\label{zk5hombas}
\begin{center}
\begin{tabular}{c|c}
$H_3(\zk)$ & $H_4(\zk)$\\
\hline
$S_1 D_3 + D_1 S_3$ & $ D_2 S_4 S_5 - S_2 S_4 D_5$\\
\hline
$S_2 D_4 + D_2 S_4$ & $  - S_1 S_3 D_5 - S_1 D_3 S_5 $\\
\hline
$S_3 D_5 + D_3 S_5$ & $ - D_1 S_2 S_4 + S_1 S_2 D_4$ \\
\hline
$S_4 D_1 + D_4 S_1$ & $ - D_2 S_3 S_5  + S_2 S_3 D_5$ \\
\hline
$S_5 D_2 + D_5 S_2$ & $ D_1 S_3 S_4 - S_1 S_3 D_4$ \\
\end{tabular}
\end{center}
\end{table}

\subsection*{Pontryagin algebra}
The canonical element $\mu_i \in H_1 (\varOmega (\mathbb{C}P^\infty)^\mathcal{K})$ is the Hurewicz image of the homotopy class of the map $S^1 \rightarrow \varOmega(\mathbb{C}P^\infty)^\mathcal{K}$ adjoint to the $i$-th coordinate inclusion $S^2 \hookrightarrow (\mathbb{C}P^\infty)^\mathcal{K}$. We denote the latter map by $\widehat{\mu}_i$ and denote its adjoint by $\mu_i \colon S^1 \rightarrow \varOmega(\mathbb{C}P^\infty)^\mathcal{K}$ (therefore using the same notation for the map and its Hurewicz image, but this will cause no confusion).
The commutator $[\mu_i, \mu_j] = \mu_i \mu_j + \mu_j \mu_i$ in the Pontryagin algebra $H_* (\varOmega (\mathbb{C}P^\infty)^\mathcal{K})$ is adjoint to the Whitehead product $[\widehat{\mu}_i, \widehat{\mu}_j] \colon S^3 \rightarrow (\mathbb{C}P^\infty)^\mathcal{K}$. It lifts to $\zk$ via the homotopy fibration $\mathcal Z_\mathcal{K} \rightarrow (\mathbb{C}P^\infty)^\mathcal{K} \rightarrow (\mathbb{C}P^\infty)^m$, and therefore we can view $[\mu_i, \mu_j]$ as an element of $H_*(\varOmega \zk)$. We consider iterated Whitehead products $[\widehat{\mu}_{i_1}, [\widehat{\mu}_{i_2}, \cdots [\widehat{\mu}_{i_{k-1}}, \widehat{\mu}_{i_k}]\cdots]] \colon S^{k+1} \to \zk \rightarrow (\mathbb{C}P^\infty)^\mathcal{K}, k \geqslant 2$, and iterated commutators $[\mu_{i_1}, [\mu_{i_2}, \cdots [\mu_{i_{k-1}}, \mu_{i_k}]\cdots]] \in H_k(\varOmega \zk) \subset H_k(\varOmega (\mathbb{C}P^\infty)^\mathcal{K})$ similarly.

\emph{Warning}: when we view $[\widehat{\mu}_{i_1}, [\widehat{\mu}_{i_2}, \cdots [\widehat{\mu}_{i_{k-1}}, \widehat{\mu}_{i_k}]\cdots]]$ as an element in $\pi_{k+1}(\zk)$, it is \emph{not} a Whitehead product, because there are no elements $\widehat{\mu}_i$ in $\pi_2(\zk)$. These elements belong to $\pi_2((\mathbb{C}P^\infty)^\mathcal{K})$; it is their Whitehead products that lift to $\pi_{k+1}(\zk)$.

Let $h \colon \pi_k (\zk) \to H_k(\zk)$ denote the Hurewicz homomorphism. From the description above it easy to deduce the following relation between iterated commutators and cellular chains:

\begin{lemma}\label{lemcomhu}
The Hurewicz image $h([ \widehat{\mu}_{i_1}, [ \widehat{\mu}_{i_2}, \cdots [ \widehat{\mu}_{i_{k-1}}, \widehat{\mu}_{i_{k}} ] \cdots ] ]) \in H_{k+1} (\zk)$ is represented by the cellular chain $S_{i_1} \cdots D_{i_{k-1}} S_{i_k} + S_{i_1} \cdots S_{i_{k-1}} D_{i_k}$.
\end{lemma}

\begin{table}[h]
\caption{Коммутаторы.}\label{comm5}
\begin{center}
\begin{tabular}{c|c}
$H_2 (\varOmega \zk)$ & $H_3 (\varOmega \zk)$\\
\hline
$[\mu_3, \mu_1]$ & $ [\mu_4, [\mu_5, \mu_2]] $\\
\hline
$[\mu_4, \mu_2]$ & $ - [\mu_1, [\mu_5, \mu_3]] $\\
\hline
$[\mu_5, \mu_3]$ & $ - [\mu_2, [\mu_4, \mu_1]] $\\
\hline
$[\mu_4, \mu_1]$ & $ - [\mu_3, [\mu_5, \mu_2]] $\\
\hline
$[\mu_5, \mu_2]$ & $ [\mu_3, [\mu_4, \mu_1]] $\\
\end{tabular}
\end{center}
\end{table}

Now we write the commutators corresponding to the basis cellular chains:
\begin{gather*}
S_i D_j + D_i S_j \leftrightarrow [\mu_i, \mu_j]\\
S_i S_j D_k + S_i D_j S_k \leftrightarrow [\mu_i, [\mu_j, \mu_k]].
\end{gather*}
These commutators are given in Table~\ref{comm5}.

{\samepage
\begin{theorem}\label{thm5} Let $\mathcal K$ be the boundary of pentagon.
\begin{itemize}
\item[(a)] There is a homotopy equivalence $\mathcal{Z}_\mathcal{K} \simeq (S^3 \times S^4)^{\#5}$, where $(S^3 \times S^4)^{\#5}$ is a connected sum of $5$ copies $S^3 \times S^4$.
\item[(b)]$H_* (\varOmega \mathcal{Z}_\mathcal{K})$ is an algebra with $ 10 $ generators:
\[
\alpha_1=[\mu_3,\mu_1], \alpha_2=[\mu_4,\mu_1], \alpha_3=[\mu_4,\mu_2], \alpha_4=[\mu_5,\mu_2], \alpha_5=[\mu_5, \mu_3],
\]
\[
\beta_1=[\mu_4,[\mu_5,\mu_2]], \beta_2=[\mu_3,[\mu_5,\mu_2]], \beta_3=[\mu_1,[\mu_5,\mu_3]], \beta_4=[\mu_3,[\mu_4,\mu_1]],
\]
\[
\beta_5=[\mu_2,[\mu_4,\mu_1]],
\]
that satisfy a single relation:
\begin{equation}\label{5sot}
- [\alpha_1,\beta_1] + [\alpha_2,\beta_2] + [\alpha_3,\beta_3] - [\alpha_4,\beta_4] + [\alpha_5,\beta_5] = 0,
\end{equation}
where $\deg \alpha_i=2, \deg \beta_i=3$, $[\alpha,\beta]=\alpha \beta-(-1)^{\deg \alpha \cdot \deg \beta}\beta \alpha$.
\end{itemize}
\end{theorem}
}
Each summand $[\alpha_i, \beta_i]$ in \eqref{5sot} corresponds to a single row in Table~\ref{comm5}.
\begin{proof}[Proof of Theorem~\ref{thm5}]
We first derive the relation~\eqref{5sot}. The $ 10 $ generators $ \alpha_1, \ldots, \alpha_5$, $\beta_1, \ldots, \beta_5 $ are exactly those given by Theorem~\ref {TP} in the case when $ \mathcal K $ is the boundary of pentagon. Geometrically, the generator $ \alpha_1 \in H_2 (\varOmega \zk) $ is given by the map $S^2 \rightarrow \varOmega \zk$  which is adjoint to the lift $S^3 \rightarrow \zk$ of the Whitehead product $[\widehat{\mu}_3, \widehat{\mu}_1]$, as described in the diagram
\[
  \xymatrix{
  &S^3 \ar^{[\widehat{\mu}_3, \widehat{\mu}_1]}[d] \ar@{-->}[dl]\\
  \zk \ar[r] & (\mathbb{C}P^\infty)^\mathcal{K} \ar[r] & (\mathbb{C}P^\infty)^m
  }
\]
Similarly, the generator $ \beta_1 \in H_3 (\varOmega \zk) $ is given by the map $S^3 \rightarrow \varOmega \zk$ adjoint to the lift $S^4 \rightarrow \zk$ in the diagram
\[
  \xymatrix{
  &S^4 \ar^{[\widehat{\mu}_4, [\widehat{\mu}_5, \widehat{\mu}_2]]}[d] \ar@{-->}[dl]\\
  \zk \ar[r] & (\mathbb{C}P^\infty)^\mathcal{K} \ar[r] & (\mathbb{C}P^\infty)^m
  }
\]


We expand each commutator in the tensor algebra $T\langle \mu_1, \ldots, \mu_5 \rangle$ using \cite{Wolfram}:
\[
\begin{array}{ll}
[\alpha_1,\beta_1]=& -\mu_1 \mu_3 \mu_2 \mu_5 \mu_4 + \mu_1 \mu_3 \mu_4 \mu_2 \mu_5 + \mu_1 \mu_3 \mu_4 \mu_5 \mu_2 - \mu_1 \mu_3 \mu_5 \mu_2 \mu_4 - \\
& - \mu_3 \mu_1 \mu_2 \mu_5 \mu_4 + \mu_3 \mu_1 \mu_4 \mu_2 \mu_5 + \mu_3 \mu_1 \mu_4 \mu_5 \mu_2 - \mu_3 \mu_1 \mu_5 \mu_2 \mu_4 + \\
& + \mu_2 \mu_5 \mu_4 \mu_1 \mu_3 - \mu_4 \mu_2 \mu_5 \mu_1 \mu_3 - \mu_4 \mu_5 \mu_2 \mu_1 \mu_3 + \mu_5 \mu_2 \mu_4 \mu_1 \mu_3 + \\
& + \mu_2 \mu_5 \mu_4 \mu_3 \mu_1 - \mu_4 \mu_2 \mu_5 \mu_3 \mu_1 - \mu_4 \mu_5 \mu_2 \mu_3 \mu_1 + \mu_5 \mu_2 \mu_4 \mu_3 \mu_1,\\
\end{array}
\]
\[
\begin{array}{ll}
[\alpha_2,\beta_2]=&-\mu_1 \mu_4 \mu_2 \mu_5 \mu_3 + \mu_1 \mu_4 \mu_3 \mu_2 \mu_5 + \mu_1 \mu_4 \mu_3 \mu_5 \mu_2 - \mu_1 \mu_4 \mu_5 \mu_2 \mu_3 - \\
& - \mu_4 \mu_1 \mu_2 \mu_5 \mu_3 + \mu_4 \mu_1 \mu_3 \mu_2 \mu_5 + \mu_4 \mu_1 \mu_3 \mu_5 \mu_2 - \mu_4 \mu_1 \mu_5 \mu_2 \mu_3 + \\
& + \mu_2 \mu_5 \mu_3 \mu_1 \mu_4 - \mu_3 \mu_2 \mu_5 \mu_1 \mu_4 - \mu_3 \mu_5 \mu_2 \mu_1 \mu_4 + \mu_5 \mu_2 \mu_3 \mu_1 \mu_4 + \\
& + \mu_2 \mu_5 \mu_3 \mu_4 \mu_1 - \mu_3 \mu_2 \mu_5 \mu_4 \mu_1 - \mu_3 \mu_5 \mu_2 \mu_4 \mu_1 + \mu_5 \mu_2 \mu_3 \mu_4 \mu_1,
\end{array}
\]
\[
\begin{array}{ll}
[\alpha_3,\beta_3]=& + \mu_2 \mu_4 \mu_1 \mu_3 \mu_5 + \mu_2 \mu_4 \mu_1 \mu_5 \mu_3 - \mu_2 \mu_4 \mu_3 \mu_5 \mu_1 - \mu_2 \mu_4 \mu_5 \mu_3 \mu_1 + \\
& + \mu_4 \mu_2 \mu_1 \mu_3 \mu_5 + \mu_4 \mu_2 \mu_1 \mu_5 \mu_3 - \mu_4 \mu_2 \mu_3 \mu_5 \mu_1 - \mu_4 \mu_2 \mu_5 \mu_3 \mu_1 - \\
& - \mu_1 \mu_3 \mu_5 \mu_2 \mu_4 - \mu_1 \mu_5 \mu_3 \mu_2 \mu_4 + \mu_3 \mu_5 \mu_1 \mu_2 \mu_4 + \mu_5 \mu_3 \mu_1 \mu_2 \mu_4 - \\
& - \mu_1 \mu_3 \mu_5 \mu_4 \mu_2 - \mu_1 \mu_5 \mu_3 \mu_4 \mu_2 + \mu_3 \mu_5 \mu_1 \mu_4 \mu_2 + \mu_5 \mu_3 \mu_1 \mu_4 \mu_2,
\end{array}
\]
\[
\begin{array}{ll}
[\alpha_4,\beta_4]=&-\mu_2 \mu_5 \mu_1 \mu_4 \mu_3 + \mu_2 \mu_5 \mu_3 \mu_1 \mu_4 + \mu_2 \mu_5 \mu_3 \mu_4 \mu_1 - \mu_2 \mu_5 \mu_4 \mu_1 \mu_3 - \\
& - \mu_5 \mu_2 \mu_1 \mu_4 \mu_3 + \mu_5 \mu_2 \mu_3 \mu_1 \mu_4 + \mu_5 \mu_2 \mu_3 \mu_4 \mu_1 - \mu_5 \mu_2 \mu_4 \mu_1 \mu_3 + \\
& + \mu_1 \mu_4 \mu_3 \mu_2 \mu_5 - \mu_3 \mu_1 \mu_4 \mu_2 \mu_5 - \mu_3 \mu_4 \mu_1 \mu_2 \mu_5 + \mu_4 \mu_1 \mu_3 \mu_2 \mu_5 + \\
& + \mu_1 \mu_4 \mu_3 \mu_5 \mu_2 - \mu_3 \mu_1 \mu_4 \mu_5 \mu_2 - \mu_3 \mu_4 \mu_1 \mu_5 \mu_2 + \mu_4 \mu_1 \mu_3 \mu_5 \mu_2,
\end{array}
\]
\[
\begin{array}{ll}
[\alpha_5,\beta_5]=&-\mu_3 \mu_5 \mu_1 \mu_4 \mu_2 + \mu_3 \mu_5 \mu_2 \mu_1 \mu_4 + \mu_3 \mu_5 \mu_2 \mu_4 \mu_1 - \mu_3 \mu_5 \mu_4 \mu_1 \mu_2 - \\
& - \mu_5 \mu_3 \mu_1 \mu_4 \mu_2 + \mu_5 \mu_3 \mu_2 \mu_1 \mu_4 + \mu_5 \mu_3 \mu_2 \mu_4 \mu_1 - \mu_5 \mu_3 \mu_4 \mu_1 \mu_2 + \\
& + \mu_1 \mu_4 \mu_2 \mu_3 \mu_5 - \mu_2 \mu_1 \mu_4 \mu_3 \mu_5 - \mu_2 \mu_4 \mu_1 \mu_3 \mu_5 + \mu_4 \mu_1 \mu_2 \mu_3 \mu_5 + \\
& + \mu_1 \mu_4 \mu_2 \mu_5 \mu_3 - \mu_2 \mu_1 \mu_4 \mu_5 \mu_3 - \mu_2 \mu_4 \mu_1 \mu_5 \mu_3 + \mu_4 \mu_1 \mu_2 \mu_5 \mu_3.
\end{array}
\]
Using the commutation relations $\mu_1 \mu_2 = -\mu_2 \mu_1$, $\mu_2 \mu_3 = -\mu_3 \mu_2$, $\mu_3 \mu_4 = -\mu_4 \mu_3$, $\mu_4 \mu_5 = -\mu_5 \mu_4$, $\mu_1 \mu_5 = -\mu_5 \mu_1$ in the algebra $H_*(\varOmega (\mathbb{C}P^\infty)^\mathcal{K})$ (see Theorem~\ref{gptw12}) we reduce each summand on the right hand side to the canonical form as follows: take $ \mu_j $ with the minimal index $ j $ (in our case $ \mu_1 $) and move it to the left as far as possible using the commutation relations. Next, take $ \mu_2 $ and move it to the left as far as possible without using the commutation relations with $ \mu_1 $. Proceed in this fashion until we reach $ \mu_5 $. For example, the canonical form of the monomial $ \mu_4 \mu_2 \mu_3 \mu_5 \mu_1 $ is $ \mu_3 \mu_4 \mu_1 \mu_2 \mu_5 $. As a result, we get
\[
\begin{array}{ll}
[\alpha_1, \beta_1] =& -\mu_1 \mu_2 \mu_3 \mu_4 \mu_5 + \mu_1 \mu_4 \mu_2 \mu_3 \mu_5 + \mu_1 \mu_3 \mu_4 \mu_5 \mu_2 - \mu_1 \mu_3 \mu_5 \mu_2 \mu_4 + \\
&+ \mu_3 \mu_1 \mu_2 \mu_4 \mu_5 + \mu_3 \mu_1 \mu_4 \mu_2 \mu_5 + \mu_3 \mu_1 \mu_4 \mu_5 \mu_2 - \mu_3 \mu_1 \mu_5 \mu_2 \mu_4 + \\
&+ \mu_2 \mu_4 \mu_1 \mu_5 \mu_3 - \mu_4 \mu_1 \mu_2 \mu_5 \mu_3 - \mu_4 \mu_1 \mu_5 \mu_2 \mu_3 + \mu_5 \mu_2 \mu_4 \mu_1 \mu_3 - \\
&- \mu_2 \mu_5 \mu_3 \mu_4 \mu_1 - \mu_4 \mu_2 \mu_5 \mu_3 \mu_1 - \mu_5 \mu_3 \mu_4 \mu_1 \mu_2 - \mu_5 \mu_2 \mu_3 \mu_4 \mu_1,
\end{array}
\]

\[
\begin{array}{ll}
[\alpha_2, \beta_2] =& -\mu_1 \mu_4 \mu_2 \mu_5 \mu_3 - \mu_1 \mu_4 \mu_2 \mu_3 \mu_5 - \mu_1 \mu_3 \mu_4 \mu_5 \mu_2 - \mu_1 \mu_4 \mu_5 \mu_2 \mu_3 - \\
& - \mu_4 \mu_1 \mu_2 \mu_5 \mu_3 - \mu_4 \mu_1 \mu_2 \mu_3 \mu_5 + \mu_4 \mu_1 \mu_3 \mu_5 \mu_2 - \mu_4 \mu_1 \mu_5 \mu_2 \mu_3 + \\
& + \mu_2 \mu_5 \mu_3 \mu_1 \mu_4 + \mu_3 \mu_1 \mu_2 \mu_4 \mu_5 - \mu_3 \mu_1 \mu_5 \mu_2 \mu_4 + \mu_5 \mu_3 \mu_1 \mu_2 \mu_4 + \\
& + \mu_2 \mu_5 \mu_3 \mu_4 \mu_1 + \mu_2 \mu_3 \mu_4 \mu_1 \mu_5 - \mu_3 \mu_5 \mu_2 \mu_4 \mu_1 + \mu_5 \mu_2 \mu_3 \mu_4 \mu_1,
\end{array}
\]

\[
\begin{array}{ll}
[\alpha_3, \beta_3] =& + \mu_2 \mu_4 \mu_1 \mu_3 \mu_5 + \mu_2 \mu_4 \mu_1 \mu_5 \mu_3 - \mu_2 \mu_3 \mu_4 \mu_1 \mu_5 - \mu_2 \mu_5 \mu_3 \mu_4 \mu_1 - \\
& - \mu_4 \mu_1 \mu_2 \mu_3 \mu_5 - \mu_4 \mu_1 \mu_2 \mu_5 \mu_3 - \mu_3 \mu_4 \mu_1 \mu_2 \mu_5 - \mu_4 \mu_2 \mu_5 \mu_3 \mu_1 - \\
& - \mu_1 \mu_3 \mu_5 \mu_2 \mu_4 + \mu_1 \mu_5 \mu_2 \mu_3 \mu_4 - \mu_3 \mu_1 \mu_5 \mu_2 \mu_4 + \mu_5 \mu_3 \mu_1 \mu_2 \mu_4 + \\
& + \mu_1 \mu_3 \mu_4 \mu_5 \mu_2 + \mu_1 \mu_4 \mu_5 \mu_2 \mu_3 + \mu_3 \mu_1 \mu_4 \mu_5 \mu_2 + \mu_5 \mu_3 \mu_1 \mu_4 \mu_2,
\end{array}
\]

\[
\begin{array}{ll}
[\alpha_4, \beta_4] =& + \mu_1 \mu_2 \mu_5 \mu_3 \mu_4 + \mu_2 \mu_5 \mu_3 \mu_1 \mu_4 + \mu_2 \mu_5 \mu_3 \mu_4 \mu_1 - \mu_2 \mu_4 \mu_1 \mu_5 \mu_3 + \\
& + \mu_1 \mu_5 \mu_2 \mu_3 \mu_4 + \mu_5 \mu_3 \mu_1 \mu_2 \mu_4 + \mu_5 \mu_2 \mu_3 \mu_4 \mu_1 - \mu_5 \mu_2 \mu_4 \mu_1 \mu_3 - \\
& - \mu_1 \mu_4 \mu_2 \mu_3 \mu_5 - \mu_3 \mu_1 \mu_4 \mu_2 \mu_5 - \mu_3 \mu_4 \mu_1 \mu_2 \mu_5 - \mu_4 \mu_1 \mu_2 \mu_3 \mu_5 - \\
& - \mu_1 \mu_3 \mu_4 \mu_5 \mu_2 - \mu_3 \mu_1 \mu_4 \mu_5 \mu_2 - \mu_3 \mu_4 \mu_1 \mu_5 \mu_2 + \mu_4 \mu_1 \mu_3 \mu_5 \mu_2,
\end{array}
\]

\[
\begin{array}{ll}
[\alpha_5, \beta_5] =& -\mu_3 \mu_1 \mu_4 \mu_5 \mu_2 + \mu_3 \mu_1 \mu_5 \mu_2 \mu_4 + \mu_3 \mu_5 \mu_2 \mu_4 \mu_1 - \mu_3 \mu_4 \mu_1 \mu_5 \mu_2 - \\
& - \mu_5 \mu_3 \mu_1 \mu_4 \mu_2 - \mu_5 \mu_3 \mu_1 \mu_2 \mu_4 - \mu_5 \mu_2 \mu_3 \mu_4 \mu_1 - \mu_5 \mu_3 \mu_4 \mu_1 \mu_2 + \\
& + \mu_1 \mu_4 \mu_2 \mu_3 \mu_5 - \mu_1 \mu_2 \mu_3 \mu_4 \mu_5 - \mu_2 \mu_4 \mu_1 \mu_3 \mu_5 + \mu_4 \mu_1 \mu_2 \mu_3 \mu_5 + \\
& + \mu_1 \mu_4 \mu_2 \mu_5 \mu_3 + \mu_1 \mu_2 \mu_5 \mu_3 \mu_4 - \mu_2 \mu_4 \mu_1 \mu_5 \mu_3 + \mu_4 \mu_1 \mu_2 \mu_5 \mu_3.
\end{array}
\]

Summing up with the appropriate signs we obtain the required relation~\eqref{5sot}.

Now we prove statement (a). Consider the map $f \colon (S^3 \vee S^4)^{\vee 5} \to \mathcal{Z}_\mathcal{K}$ defined as the wedge of the maps corresponding to the generators $\alpha_1, \ldots, \alpha_5$, $\beta_1, \ldots, \beta_5$ as described in the beginning of the proof. Because of the relation~\eqref{5sot}, the map $f$ extends to a map from the connected sum (which differs from the wedge by a single $7$-dimensional cell), as shown in the diagram below:
\[
  \xymatrix{
  (S^3 \vee S^4)^{\vee 5} \ar@{^{(}->}[d] \ar[r]^(.65){f} & \mathcal{Z}_\mathcal{K}\\
  (S^3 \times S^4)^{\#5} \ar^{\widehat{f}}@{-->}[ur] &
  }
\]
The map $\widehat{f}$ induces an isomorphism in homology, so it is a homotopy equivalence, because all spaces are simply connected.

To finish the proof of statement (b) we need to show that~\eqref{5sot} is the only relation on the generators $\alpha_1, \ldots, \alpha_5$, $\beta_1, \ldots, \beta_5$. We constructed a homotopy equivalence between $ \mathcal Z_{\mathcal K} $ and the connected sum of the products of spheres $ X = (S ^ 3 \times S ^ 4) ^ {\# 5} $. The space $X$ is obtained from a wedge of spheres $ X = (S ^ 3 \vee S ^ 4) ^ {\vee 5} $ by attaching a $7$-dimensional cell along a sum of five Whitehead products. By passing to the Pontryagin algebra, we obtain that $H_*(\varOmega X)$ is the quotient of a free algebra on $10$ generators by a single relation, i.e. has exactly the form decribed in statement (b). The homotopy equivalence $X \rightarrow \zk$ implies an isomorphism of Pontryagin algebras, so there are no other relations except \eqref{5sot}.
%
%
\end{proof}
\section{The case of hexagon}

In this section, $ \mathcal K $ is the boundary of hexagon. That is, $\mathcal{K}$ is the simplicial complex on $[6] = \{1, 2, 3, 4, 5, 6\}$ with maximal simplices $\{1,2\}$, $\{2,3\}$, $\{3,4\}$, $\{4,5\}$, $\{5,6\}$, $\{6,1\}$. We describe the cohomology, homology and Pontryagin algebra of $\zk$.

We calculate the cohomology and homology of $\zk$ similarly to the case of pentagon. Then we describe the Pontryagin algebra $H_*(\varOmega \zk)$ by writing down a single relation on the canonical generator set of iterated commutators from Theorem~\ref{TP}. Unlike the case of pentagon, iterated commutators of the form $[ \alpha_i, \alpha_j]$ appear in the relation. These commutators correspond to the Whitehead products $[\widehat{\alpha}_i, \widehat{\alpha}_j] \colon S^5 \rightarrow \zk$, which vanish under the Hurewicz homomorphism and therefore can not be detected by the cohomology ring.

\subsection*{Cohomology}

The cohomology ring structure of $\zk$ can be described using Theorem~\ref{hzk}. The nontrivial cohomology groups are
\begin{align*}
H^0 (\zk) & \cong \widetilde{H}^{-1} (\mathcal{K}_\varnothing) \cong k,\\
H^3 (\zk) & \cong \bigoplus_{|I|=2} \widetilde{H}^0 (\mathcal{K}_I) \cong k^9,\\
H^4 (\zk) & \cong \bigoplus_{|I|=3} \widetilde{H}^0 (\mathcal{K}_I) \cong k^{16},\\
H^5 (\zk) & \cong \bigoplus_{|I|=4} \widetilde{H}^0 (\mathcal{K}_I) \cong k^9,\\
H^8 (\zk) & \cong \widetilde{H}^1 (\mathcal{K}) \cong k.
\end{align*}

A generator of $H^8(\zk)$ is represented by any monomial $u_i u_{i+1} u_{i+2} v_{i+3} v_{i+4} v_{i+5} \in \varLambda [u_1, \ldots, u_6] \otimes k[\mathcal{K}]$, where the indices are considered modulo~$6$. All these monomials represent the same cohomology classes.

We denote $ t = [u_1 u_2 u_3 v_4 v_5 v_6] \in H^8 (\zk)$ and calculate the product in the cohomology ring.
We choose bases for $H^3(\zk)$ and $H^5(\zk)$ as shown in Tables~\ref{zk6prod35},~\ref{zk6prod44}. For any given basis element of $H^3(\zk)$ there is a unique basis element of $H^5(\zk)$ such that the product of these two elements is $t$. The product of any other two elements of $H^3(\zk)$ and $H^5(\zk)$ is zero. Similarly for $H^4(\zk)$.


\begin{table}[h]
\caption{Cohomology classes and their products for $H^3(\zk)$ and $H^5(\zk)$.}\label{zk6prod35}
\begin{center}
\begin{tabular}{c|c|c}
$H^3(\zk)$ & $H^5(\zk)$ & Product \\
\hline
$ [u_1 v_3] $ & $ [u_4 u_5 u_6 v_2] $ & $t$ \\
\hline
$ [u_1 v_4] $ & $ - [u_2 u_3 u_5 v_6] + [u_2 u_3 u_6 v_5] $ & $t$ \\
\hline
$ [u_1 v_5] $ & $ [u_2 u_3 u_4 v_6] $ & $t$ \\
\hline
$ [u_2 v_4] $ & $  - [u_1 u_5 u_6 v_3] $ & $t$ \\
\hline
$ [u_2 v_5] $ & $ - [u_1 u_3 u_6 v_4] + [u_1 u_4 u_6 v_3] $ & $t$ \\
\hline
$ [u_2 v_6] $ & $ - [u_3 u_4 u_5 v_1] $ & $t$ \\
\hline
$ [u_3 v_5] $ & $ [u_1 u_2 u_6 v_4] $ & $t$ \\
\hline
$ [u_3 v_6] $ & $ [u_1 u_2 u_4 v_5] - [u_1 u_2 u_5 v_4] $ & $t$ \\
\hline
$ [u_4 v_6] $ & $ - [u_1 u_2 u_3 v_5] $ & $t$ \\
\end{tabular}
\end{center}
\end{table}
\begin{table}[h]
\caption{Cohomology classes and their products for $H^4(\zk)$.}\label{zk6prod44}
\begin{center}
\begin{tabular}{c|c|c}
$H^4(\zk)$ & $H^4(\zk)$ & Product \\
\hline
$ [u_1 u_5 v_3] $ & $ - [u_2 u_6 v_4] $ & $t$ \\
\hline
$ [u_3 u_5 v_1] $ & $ - [u_4 u_6 v_2] + [u_2 u_6 v_4] $ & $t$ \\
\hline
$ [u_2 u_3 v_6] $ & $ -[u_4 u_5 v_1] $ & $t$ \\
\hline
$ [u_5 u_6 v_2] $ & $ [u_3 u_4 v_1] $ & $t$ \\
\hline
$ [u_1 u_6 v_3] $ & $ [u_4 u_5 v_2] $ & $t$ \\
\hline
$ [u_3 u_4 v_6] $ & $ - [u_2 u_5 v_1] + [u_1 u_5 v_2] $ & $t$ \\
\hline
$ [u_5 u_6 v_3] $ & $ - [u_2 u_4 v_1] + [u_1 u_4 v_2] $ & $t$ \\
\hline
$ [u_1 u_6 v_4] $ & $ - [u_3 u_5 v_2] + [u_2 u_5 v_3] $ & $t$ \\
\end{tabular}
\end{center}
\end{table}

Note that, unlike the case of pentagon, some cohomology classes in $H^4(\zk)$ and $H^5(\zk)$ can not be represented by monomials in the Koszul algebra $\varLambda[u_1, \ldots, u_6]~\otimes~k[\mathcal{K}]$.

\subsection*{Homology}
%
The cellular cycles dual to the cohomology classes in Table~\ref{zk6prod35} are shown in Table~\ref{zk6hombas35}, and the cellular cycles dual to the cohomology classes  in Table~\ref{zk6prod44} are shown in Table~\ref{zk6hombas44}.
\begin{table}[h]
\caption{Cellular cycles representing basis homology classes for $H_3(\zk)$ and $H_5(\zk)$.}\label{zk6hombas35}
\begin{center}
\begin{tabular}{c|c}
$H_3(\zk)$ & $H_5(\zk)$ \\
\hline
$S_1 D_3 + D_1 S_3$ & $ D_2 S_4 S_5 S_6 + S_2 S_4 S_5 D_6 $ \\
\hline
$S_1 D_4 + D_1 S_4$ & $ - S_2 S_3 S_5 D_6 - D_2 S_3 S_5 S_6 $ \\
\hline
$S_1 D_5 + D_1 S_5$ & $ D_2 S_3 S_4 S_6 + S_2 S_3 S_4 D_6 $ \\
\hline
$S_2 D_4 + D_2 S_4$ & $ S_1 S_3 S_5 D_6 - S_1 D_3 S_5 S_6 $ \\
\hline
$S_2 D_5 + D_2 S_5$ & $ S_1 D_3 S_4 S_6 - S_1 S_3 S_4 D_6 $ \\
\hline
$S_2 D_6 + D_2 S_6$ & $ - D_1 S_3 S_4 S_5 - S_1 S_3 S_4 D_5 $ \\
\hline
$S_3 D_5 + D_3 S_5$ & $ S_1 S_2 D_4 S_6 + S_1 S_2 S_4 D_6 $ \\
\hline
$S_3 D_6 + D_3 S_6$ & $ D_1 S_2 S_4 S_5 + S_1 S_2 S_4 D_5 $ \\
\hline
$S_4 D_6 + D_4 S_6$ & $ - S_1 S_2 S_3 D_5 - D_1 S_2 S_3 S_5 $ \\
\end{tabular}
\end{center}
\end{table}
\begin{table}[h]
\caption{Cellular cycles representing basis homology classes for $H_4(\zk)$.}\label{zk6hombas44}
\begin{center}
\begin{tabular}{c|c}
$H_4(\zk)$ & $H_4(\zk)$ \\
\hline
$ S_1 S_3 D_5 + S_1 D_3 S_5$ & $ - D_2 S_4 S_6 - S_2 D_4 S_6 $ \\
\hline
$ - S_1 S_3 D_5 + D_1 S_3 S_5$ & $ S_2 S_4 D_6 - D_2 S_4 S_6 $ \\
\hline
$ S_2 S_3 D_6 - D_2 S_3 S_6$ & $ S_1 S_4 D_5 - D_1 S_4 S_5 $ \\
\hline
$ D_2 S_5 S_6 - S_2 S_5 D_6$ & $ - S_1 S_3 D_4 + D_1 S_3 S_4 $ \\
\hline
$ S_1 D_3 S_6 + S_1 S_3 D_6$ & $ - S_2 S_4 D_5 + D_2 S_4 S_5 $ \\
\hline
$ S_3 S_4 D_6 - D_3 S_4 S_6$ & $ S_1 S_2 D_5 - D_1 S_2 S_5 $ \\
\hline
$ - S_3 S_5 D_6 + D_3 S_5 S_6$ & $ - D_1 S_2 S_4 + S_1 S_2 D_4 $ \\
\hline
$ S_1 S_4 D_6 + S_1 D_4 S_6$ & $ - D_2 S_3 S_5 + S_2 S_3 D_5 $ \\
\end{tabular}
\end{center}
\end{table}

{\samepage

\subsection*{Pontryagin algebra}
Similarly to the case of pentagon we use Lemma~\ref{lemcomhu} to write down the commutators correposponding to the cellular chains in Tables~\ref{zk6hombas35},~\ref{zk6hombas44}. These are given in the first two columns of Tables~\ref{com6_35},~\ref{com6_44}. However, unlike the case of pentagon the Hurewicz homomorphism $\pi_5(\zk) \rightarrow H_5(\zk)$ is not injective; its kernel contains commutators of the form $[[\mu_i, \mu_j], [\mu_k, \mu_l]]$. These ``additional" commutators are given in the third column of Table~\ref{com6_35} with some unknown coefficients~$k_i$.

\begin{table}[h]
\caption{Commutators in $H_2(\varOmega\zk)$ and $H_4(\varOmega\zk)$.}\label{com6_35}
\begin{center}
\begin{tabular}{c|c|c}
$H_2(\varOmega \zk)$ & $H_4(\varOmega \zk)$ & Additional commutators \\
\hline
$[\mu_3, \mu_1]$ & $ [\mu_4, [\mu_5, [\mu_6, \mu_2]]]$ & $k_1 [[\mu_2, \mu_5], [\mu_4, \mu_6]]$ \\
\hline
$[\mu_4, \mu_1]$ & $ - [\mu_3, [\mu_5, [\mu_6, \mu_2]]]$ & $k_2 [[\mu_5, \mu_3], [\mu_6, \mu_2]] + k_3 [[\mu_6, \mu_3], [\mu_5, \mu_2]]$ \\
\hline
$[\mu_5, \mu_1]$ & $ [\mu_3, [\mu_4, [\mu_6, \mu_2]]]$ & $k_4 [[\mu_4, \mu_2], [\mu_6, \mu_3]]$ \\
\hline
$[\mu_4, \mu_2]$ & $ - [\mu_1, [\mu_5, [\mu_6, \mu_3]]]$ & $k_5 [[\mu_6, \mu_3], [\mu_5, \mu_1]]$ \\
\hline
$[\mu_5, \mu_2]$ & $ [\mu_1, [\mu_4, [\mu_6, \mu_3]]]$ & $k_6 [[\mu_4, \mu_1], [\mu_6, \mu_3]] + k_7 [[\mu_6, \mu_4], [\mu_3, \mu_1]]$ \\
\hline
$[\mu_6, \mu_2]$ & $ - [\mu_3, [\mu_4, [\mu_5, \mu_1]]]$ & $k_8 [[\mu_5, \mu_3], [\mu_4, \mu_1]]$ \\
\hline
$[\mu_5, \mu_3]$ & $[\mu_1, [\mu_2, [\mu_6, \mu_4]]]$ & $k_9 [[\mu_4, \mu_1], [\mu_6, \mu_2]]$ \\
\hline
$[\mu_6, \mu_3]$ & $[\mu_2, [\mu_4, [\mu_5, \mu_1]]]$ & $k_{10} [[\mu_4, \mu_2], [\mu_5, \mu_1]] + k_{11} [[\mu_4, \mu_1], [\mu_5, \mu_2]]$ \\
\hline
$[\mu_6, \mu_4]$ & $ - [\mu_2, [\mu_3, [\mu_5, \mu_1]]]$ & $k_{12} [[\mu_5, \mu_2], [\mu_3, \mu_1]]$ \\
\end{tabular}
\end{center}
\end{table}
\begin{table}[h]
\caption{Commutators in $H_3(\varOmega\zk)$.}\label{com6_44}
\begin{center}
\begin{tabular}{c|c}
$H_3(\varOmega \zk)$ & $H_3(\varOmega \zk)$ \\
\hline
$ [\mu_1, [\mu_5, \mu_3]] $ & $ [\mu_6, [\mu_4, \mu_2]] $ \\
\hline
$ [\mu_3, [\mu_5, \mu_1]] $ & $ - [\mu_4, [\mu_6, \mu_2]] $ \\
\hline
$ - [\mu_3, [\mu_6, \mu_2]] $ & $ - [\mu_4, [\mu_5, \mu_1]] $ \\
\hline
$ [\mu_5, [\mu_6, \mu_2]] $ & $ [\mu_3, [\mu_4, \mu_1]] $ \\
\hline
$ [\mu_1, [\mu_6, \mu_3]] $ & $ [\mu_4, [\mu_5, \mu_2]] $ \\
\hline
$ - [\mu_4, [\mu_6, \mu_3]] $ & $ - [\mu_2, [\mu_5, \mu_1]] $ \\
\hline
$ [\mu_5, [\mu_6, \mu_3]] $ & $ - [\mu_2, [\mu_4, \mu_1]] $ \\
\hline
$ [\mu_1, [\mu_6, \mu_4]] $ & $ - [\mu_3, [\mu_5, \mu_2]] $ \\
\end{tabular}
\end{center}
\end{table}
}

{\samepage
\begin{theorem}\label{thm6}
Let $\mathcal{K}$ be the boundary of hexagon.
\begin{itemize}
\item[(a)] There is a homotopy equivalence:
$$\mathcal{Z}_\mathcal{K} \simeq (S^3 \times S^5)^{\#9} \# (S^4 \times S^4)^{\#8}.$$
\item[(b)] $H_* (\varOmega \mathcal{Z}_\mathcal{K})$ is an algebra with $ 34 $ generators:
\[
\alpha_1=[\mu_3, \mu_1], \alpha_2=[\mu_4, \mu_1], \alpha_3=[\mu_5, \mu_1], \alpha_4=[\mu_4, \mu_5], \alpha_5=[\mu_5, \mu_2],
\]
\[
\alpha_6=[\mu_6, \mu_2], \alpha_7=[\mu_5, \mu_3], \alpha_8=[\mu_6, \mu_3], \alpha_9=[\mu_6, \mu_4],
\]
\[
\beta_1=[\mu_1, [\mu_5, \mu_3]], \beta_2=[\mu_3, [\mu_5, \mu_1]], \beta_3=[\mu_3, [\mu_6, \mu_2]], \beta_4=[\mu_5, [\mu_6, \mu_2]],
\]
\[
\beta_5=[\mu_1, [\mu_6, \mu_3]], \beta_6=[\mu_4, [\mu_6, \mu_3]], \beta_7=[\mu_5, [\mu_6, \mu_3]], \beta_8=[\mu_1, [\mu_6, \mu_4]],
\]
\[
\delta_1=[\mu_6, [\mu_4, \mu_2]], \delta_2=[\mu_4, [\mu_6, \mu_2]], \delta_3=[\mu_4, [\mu_5, \mu_1]], \delta_4=[\mu_3, [\mu_4, \mu_1]],
\]
\[
\delta_5=[\mu_4, [\mu_5, \mu_2]], \delta_6=[\mu_2, [\mu_5, \mu_1]], \delta_7=[\mu_2, [\mu_4, \mu_1]], \delta_8=[\mu_3, [\mu_5, \mu_2]],
\]
\[
\gamma_1=[\mu_4, [\mu_5, [\mu_6, \mu_2]]], \gamma_2=[\mu_3, [\mu_5, [\mu_6, \mu_2]]], \gamma_3=[\mu_3, [\mu_4, [\mu_6, \mu_2]]],
\]
\[
\gamma_4=[\mu_1, [\mu_5, [\mu_6, \mu_3]]],  \gamma_5=[\mu_1, [\mu_4, [\mu_6, \mu_3]]], \gamma_6=[\mu_3, [\mu_4, [\mu_5, \mu_1]]],
\]
\[
\gamma_7=[\mu_1, [\mu_2, [\mu_6, \mu_4]]], \gamma_8=[\mu_2, [\mu_4, [\mu_5, \mu_1]]], \gamma_9=[\mu_2, [\mu_3, [\mu_5, \mu_1]]],
\]
that satisfy a single relation:
\begin{equation}\label{rel6}
\sum_{i=1}^9{ [\alpha_i, \gamma_i^\prime] } + \sum_{j=1}^8{ \sigma_j \cdot [\beta_j, \delta_j] } = 0,
\end{equation}
\newpage
where
\[
\deg \alpha_i=2, \deg \beta_i=\deg \delta_i=3, \deg \gamma_i=4,
\]
\[
\gamma_1^\prime = - \gamma_1 + [\alpha_5, \alpha_9], \gamma_2^\prime = \gamma_2 + [\alpha_7, \alpha_6] - [\alpha_8, \alpha_5], \gamma_3^\prime = - \gamma_3 + [\alpha_4, \alpha_8],
\]
\[
\gamma_4^\prime = \gamma_4 - [\alpha_8, \alpha_3], \gamma_5^\prime = - \gamma_5 + [\alpha_2, \alpha_8] - [\alpha_9, \alpha_1], \gamma_6^\prime = \gamma_6 - [\alpha_7, \alpha_2],
\]
\[
\gamma_7^\prime = - \gamma_7 + 0 \cdot [\alpha_2, \alpha_6], \gamma_8^\prime = - \gamma_8 + 0 \cdot [\alpha_4, \alpha_3] + 0 \cdot [\alpha_2, \alpha_5], \gamma_9^\prime = \gamma_9 + 0 \cdot [\alpha_5, \alpha_1],
\]
\[
    \begin{matrix}
    \sigma_j & =
    & \left\{
    \begin{matrix}
    -1, & j \in \{ 2, 7, 8 \} \\
    1, & j \in \{ 1, 3, 4, 5, 6 \}
    \end{matrix} \right..
    \end{matrix}
\]
\end{itemize}
\end{theorem}
}

\begin{proof}
The scheme of proof is similar to the case of pentagon (Theorem~\ref{thm5}). The main step is to deduce the relation~\eqref{rel6}.
We expand each commutator $[\mu_i, \mu_j]$, $[\mu_i, [\mu_j, \mu_k]]$, $[\mu_l, [\mu_k, [\mu_i, \mu_j]]]$ from the definition of $\alpha_i$, $\beta_i$, $\gamma_i$, $\delta_i$ in the tensor algebra $T\langle \mu_1, \ldots, \mu_6 \rangle$ and reduce each monomial to the canonical form. Then do the same with the additional commutators from the third column of Table~\ref{com6_35}. Summing up the resulting expressions and equating to zero we get a system of linear equations on the unknown coefficients $k_i$. Solving this system using~\cite{Wolfram} we obtain the following relation on the iterated commutators:
{\small \begin{align*}
&- [[\mu_3, \mu_1], [\mu_4, [\mu_5, [\mu_6, \mu_2]]]] + [[\mu_3, \mu_1], [[\mu_2, \mu_5], [\mu_4, \mu_6]]] + [[\mu_4, \mu_1], [\mu_3, [\mu_5, [\mu_6, \mu_2]]]] +\\
&+ [[\mu_4, \mu_1], [[\mu_5, \mu_3], [\mu_6, \mu_2]]] - [[\mu_4, \mu_1], [[\mu_6, \mu_3], [\mu_5, \mu_2]]] - [[\mu_5, \mu_1], [\mu_3, [\mu_4, [\mu_6, \mu_2]]]] +\\
&+ [[\mu_5, \mu_1], [[\mu_4, \mu_2], [\mu_6, \mu_3]]] + [[\mu_4, \mu_2], [\mu_1, [\mu_5, [\mu_6, \mu_3]]]] - [[\mu_4, \mu_2], [[\mu_6, \mu_3], [\mu_5, \mu_1]]] -\\
&- [[\mu_5, \mu_2], [\mu_1, [\mu_4, [\mu_6, \mu_3]]]] + [[\mu_5, \mu_2], [[\mu_4, \mu_1], [\mu_6, \mu_3]]] - [[\mu_5, \mu_2], [[\mu_6, \mu_4], [\mu_3, \mu_1]]] +\\
&+ [[\mu_6, \mu_2], [\mu_3, [\mu_4, [\mu_5, \mu_1]]]] - [[\mu_6, \mu_2], [[\mu_5, \mu_3], [\mu_4, \mu_1]]] - [[\mu_5, \mu_3], [\mu_1, [\mu_2, [\mu_6, \mu_4]]]] -\\
&- [[\mu_6, \mu_3], [\mu_2, [\mu_4, [\mu_5, \mu_1]]]] + [[\mu_6, \mu_4], [\mu_2, [\mu_3, [\mu_5, \mu_1]]]] + [[\mu_1, [\mu_5, \mu_3]], [\mu_6, [\mu_4, \mu_2]]] -\\
&- [[\mu_3, [\mu_5, \mu_1]], [\mu_4, [\mu_6, \mu_2]]] + [[\mu_3, [\mu_6, \mu_2]], [\mu_4, [\mu_5, \mu_1]]] + [[\mu_5, [\mu_6, \mu_2]], [\mu_3, [\mu_4, \mu_1]]] +\\
&+ [[\mu_1, [\mu_6, \mu_3]], [\mu_4, [\mu_5, \mu_2]]] + [[\mu_4, [\mu_6, \mu_3]], [\mu_2, [\mu_5, \mu_1]]] - [[\mu_5, [\mu_6, \mu_3]], [\mu_2, [\mu_4, \mu_1]]] -\\
&- [[\mu_1, [\mu_6, \mu_4]], [\mu_3, [\mu_5, \mu_2]]] = 0.
\end{align*}}%
This is equivalent to~\eqref{rel6}. The rest of the proof is the same as for Theorem~\ref{thm5}: we use the relation~\eqref{rel6} to construct a map
$$
\widehat{f} \colon (S^3 \times S^5)^{\#9} \# (S^4 \times S^4)^{\#8} \rightarrow \zk
$$
and show that it is a homotopy equivalence.

Remark that signs before commutators are equivalent to signs in Table~\ref{com6_44} and opposite to signs in Table~\ref{com6_35}.
\end{proof}


\end{document}